\documentclass{amsart}
\usepackage[utf8]{inputenc} 
\usepackage[all]{xy}
\usepackage{amssymb}
\usepackage{amsmath}
\usepackage{hyperref}
\newtheorem{theorem}{Theorem}[section]
\newtheorem{lemma}[theorem]{Lemma}

\newtheorem{remark}[theorem]{Remark}

\newtheorem{examples}[theorem]{Examples}

\newtheorem*{bem}{Remark}{\it}{}

\numberwithin{equation}{section}

\newcommand{\Z}{{\mathbb Z}}

\newcommand{\Q}{{\mathbb Q}}
\newcommand{\C}{{\mathbb C}}

\newcommand{\x}{{\bf x}}

\begin{document}

\title[Some congruences for Siegel theta series]
{Some congruences for Siegel theta series}  
\author[R. Schulze-Pillot]{Rainer Schulze-Pillot} 
\begin{abstract}
  We discuss an arithmetic approach to some congruence properties of
  Siegel theta series of even positive definite unimodular quadratic forms.
\end{abstract}

\maketitle

 \section{Introduction.} There has recently been some interest in 
congruence properties of Siegel modular forms. In particular, following
B\"ocherer and Nagaoka \cite{boe-thop}, Siegel modular forms $f$ with
$p$-integral Fourier coefficients have been considered
for which $\Theta(f) \equiv 0 \bmod p$ for some prime $p$. Here for a
Siegel modular form $f$ with Fourier expansion 
$$f{(Z)} = \sum a(f,T) \exp(2\pi i {\rm tr}(TZ))$$ for $Z$ in the
Siegel upper half space $${\mathfrak H}_n=\{Z \in M_n^{\rm
  sym}(\C)\mid {\rm Im}(Z)\text{ positive definite}\}$$ one puts $$\Theta(f){(Z)} =
\sum \det(T) a(f,T) \exp(2 \pi i {\rm tr}(TZ)).$$
The map $f\mapsto \Theta(f)$ is called the theta operator, and one
writes  $\Theta(f) \equiv 0 \bmod p$ if all
the coefficients $ \det(T) a(f,T)$ of $\Theta(f)$ are in the maximal
ideal $p\Z_{(p)}$ of the local ring $\Z_{(p)}$ at $p$. 

In this note we show that unimodular lattices with an automorphism of order $p$ give
a quite natural series of examples for this behavior. In particular, we recover 
the recent result of Nagaoka and Takemori from \cite{naga-leech} for
the theta series of the Leech lattice.

I thank S. Böcherer for telling me about the problem and S. Nagaoka
and S. Takemori for showing me their preprint \cite{naga-leech}.
 
 \section{Fixed space decomposition of lattices with an automorphism of prime order} 
In this section we recall and modify some results from
\cite{hambleton-riehm,quebbemann} (see also \cite[Section 4]{nebe_automs}) and fix some notations for the rest
of the article.
 \vspace{0.3cm}

Let $V$ be an $m$-dimensional $\Q$-vector space with positive definite
quadratic form $q$ and associated symmetric bilinear form
$b(x,y)=q(x+y)-q(x)-q(y)$.
For an $r$-tuple ${\bf x}=(x_1,\ldots,x_r)\in V^r$ we denote by
$G_b({\bf x})$ its Gram matrix $(b(x_i,x_j))\in M_r^{\rm
  sym}(\Q)$ with respect to $b$ and write $q(\x)=\frac{1}{2}G_b(\x)$. 
The determinant $\det(M')$ of a lattice $M'$ of rank $r$ in $V$ with $\Z$-basis
$\x=(x_1,\ldots,x_r)$ is $\det(G_b(\x))$. By $O(V,q)$ resp. $O(M,q)$
we denote the group of isometries of $V$ resp. of a lattice $M$ of
full rank $m$ on $V$
onto itself with respect to $q$. 

Let $M$ be an integral lattice  with respect to $q$ (i.e.,
$q(M)\subseteq \Z$) on $V$ with an automorphism
$\sigma \in O(M,q)$ of order $p$, where $p\ne 2$ is a prime.
We recall that the lattice $M$ is called decomposable if it is an
orthogonal sum of proper sublattices and that the decomposition of $M$
into such a sum (if it exists) is unique \cite[Satz 2]{eichler_kristall},\cite{kneser_kristall}.
 \vspace{0.3cm}

We denote by $\Pi = \langle \sigma \rangle$ the cyclic group generated by $\sigma$ and consider 
$M$ as a module over the group ring $\Lambda = \Z[\Pi]$.

We consider the $\Z$-linear embedding
 \begin{equation*}
 \iota:\: \Lambda \longrightarrow \Gamma = \Z \oplus \Z[\zeta_p]
 \end{equation*}
given by $\iota(\sigma^i) = (1,\zeta_p^i)$; it is clearly an embedding of rings, which
extends to an isomorphism 
 \begin{equation*}
 \iota:\: \Q[\Pi] \longrightarrow \Q \Gamma = \Q \oplus \Q(\zeta_p).
 \end{equation*}
We have then
 \begin{equation*}
 \iota(\sum_{i=0}^{p-1} \alpha_i\sigma^i) = (a,\sum_{i=1}^{p-1} \beta_i \zeta^i)
 \end{equation*}
with
 \begin{equation*}
 p\alpha_0 = a+{\rm Tr}_{\Q}^{\Q(\zeta_p)} (\sum_{i=1}^{p-1} \beta_i \zeta^i) 
 \end{equation*}
and $\alpha_i = \beta_i+\alpha_0$ for $1 \leq i \leq p-1$.

In particular we have $I:=p\Z\oplus(1-\zeta)\Z[\zeta_p]\subseteq
\iota(\Lambda)$. We will identify $\Lambda$ and  $\iota(\Lambda)$ in
the sequel.
 \vspace{0.3cm}

We denote by $V_0$ the fixed space of $\sigma$ in $V$ and split $V =
V_0 \perp V_1$, so that 
 \begin{equation*}
 \sum^{p-1}_{j=0} \sigma^j|_{V_1} = 0.
 \end{equation*}
The $\Q[\Pi]$-module $V_1$ can then be viewed via 
 \begin{equation*}
 \Q(\zeta_p) \cong \Q[X]/(X^{p-1}+ \cdots + 1)
 \end{equation*}
as a $\Q(\zeta_p)$-vector space.
In particular we notice
\begin{lemma}
The dimension of $V_1$ as a vector space over $\Q$ is
divisible by $p-1$.
\end{lemma}
Moreover, $(a,b) = (a,\sum_{i=1}^{p-1}
\beta_i \zeta^i) \in \Q\Gamma=\Q\Lambda=\Q[\Pi]$  acts
on $V_0$ by multiplication with $a$, on $V_1$ by multiplication with
$b$, again identifying $\Q[\Pi]=\Q\Lambda$ and $\Q[\Gamma]$ via $\iota$.

 \begin{lemma}\label{gamma_m_splitting_lemma}
The lattice $\Gamma M \subseteq V$ splits as $\Gamma M = \tilde{M_0} \oplus \tilde{M}_1$, where
$\tilde{M}_i = \pi_i(M)$ is the orthogonal projection of $M$ onto $V_i$ ($i = 0.1$). 
 \end{lemma}

 \begin{proof}
The splitting $\Gamma = \Z \oplus \Z[\zeta_p]$ induces corresponding splittings of $V$ and
$\Gamma M$. In view of the remarks above on the action of $\Q\Gamma$ on $V_0$ and $V_1$ we see that
the splitting of $V$ is indeed the splitting $V = V_0 \perp V_1$, we write $\Gamma M =
\tilde{M}_0 \oplus \tilde{M}_1$. For $m_0 \in \tilde{M}_0$ we can write 
$m_0 = (1,0) \cdot m$ with $m \in M$, then $m_1 = (0,1) m \in \tilde{M}_1$ gives
$m = m_0 + m_1$, so $m_0 = \pi_0(m) \in \pi_0(M)$, in the same way we see $\tilde{M}_1 \subseteq
\pi_1(M)$. Conversely, given  $m_i \in \pi_i(M)$, $m_i = \pi_i(m)$ with $m \in M$, we have 
$m_i = \epsilon_im$, where $\epsilon_0 = (1,0)$, $\epsilon_1 = (0,1)
\in \Gamma$ are the orthogonal idempotents of $\Gamma$. 
 \end{proof}
 \begin{lemma}\label{intersection_lemma}
With the notation as above and $M_i = M \cap V_i$ one has
 \begin{equation*}
 p \tilde{M}_i \subseteq M_i \subseteq \tilde{M}_i \subseteq M_i^{\#} \mbox{ for } i = 0,1,
 \end{equation*}
where $M_i^\#=\{x\in V_i\mid b(x,M_i)\subseteq \Z\}$ denotes the dual lattice of $M_i$ with respect to the
symmetric bilinear form $b$.
 \end{lemma}

 \begin{proof}
$p\Gamma \subseteq \Lambda$ implies $p\tilde{M}_i \subseteq p\Gamma M \subseteq \Lambda M = M$ for
$i=0,1$, so $p\tilde{M}_i \subseteq M\cap V_i = M_i$.\\
$M_i \subseteq M = \Lambda M \subseteq \Gamma M$ implies $M_i \subseteq \tilde{M}_i$.
The inclusion $\tilde{M}_i \subseteq M_i^{\#}$ is obvious.
 \end{proof}

 \begin{lemma}\label{discriminantlemma}
With the notation as above one has $M =M_0 \perp M_1$ or $M_0$ has determinant divisible by $p$.
 \end{lemma}

 \begin{proof}
If $M \not= M_0 \perp M_1$ one has $M_0 \subsetneqq \tilde{M}_0$, and $(\tilde{M}_0:M_0)$ is a positive
power of $p$ by $p\tilde{M}_0 \subseteq M_0 \subseteq \tilde{M}_0$. But then $\det(M_0) = (M_0^{\#}:M_0)$
is also divisible by a positive power of $p$.
 \end{proof}

 \begin{theorem}\label{lattice_theorem}
With notation as above and $m_i = \dim V_i$ ($i=0,1$) one has:\\
If $M$ is indecomposable or decomposable with no proper orthogonal summand of rank $m_0$ and
determinant prime to $p$, the $\sigma$-fixed sublattice $M_0 = M\cap V_0$ of $M$ has determinant
divisible by $p$. 
 \end{theorem}

 \begin{proof}
Clear from the lemmas above.
 \end{proof}

 \section{Congruences of Theta series.}

 \begin{theorem}\label{theta_theorem}
Let $(V,q)$ be a positive definite  $m$-dimensional quadratic space
over $\Q$, let $M$ be a $q$-integral lattice on $V$ which is
unimodular with respect to $b$ (also called an even
unimodular lattice with respect to the symmetric bilinear form $b$).
Assume that $M$ has an  automorphism $\sigma\in O(M,q)$ of odd prime order $p$, let $V_0 \subseteq V$ be the
fixed space of $\sigma$, $m_0 = \dim V_0$, let (for $Z\in {\mathfrak H}_{m_0}$)
 \begin{eqnarray*}
 F_M(Z) &=&\sum_{{\bf x}=(x_1,\ldots,x_{m_0})\in M^{m_0}}\exp (2 \pi i {\rm
   tr}(q({\bf x})Z))\\ 
&=&\sum_{T\in M_{m_0}^{\rm sym}(\Z)} A(M,T) \exp (2 \pi i {\rm tr}(TZ)) \in M_{m/2} (Sp_{m_0}(\Z))
 \end{eqnarray*}
denote the degree $m_0$ theta series of $(M,q)$.

Write (as in \cite{boe-thop})
 \begin{equation*}
 \Theta F_M(Z) = \sum_{T} \det(T) A(M,T) \exp (2 \pi i {\rm tr}(TZ)).
 \end{equation*}
Then $\Theta F_M \equiv 0 \bmod p$ unless $M$ is decomposable (into an
orthogonal sum of sublattices) and has a proper orthogonal
summand of rank $m_0$.\\
In particular, if $m_0$ is not divisible by $8$, $\Theta F_M$ is congruent to zero modulo $p$.
\end{theorem}
\begin{proof}
If $T \in M_{m_0}^{\rm sym}(\Z)$ is a positive definite matrix with $A(M,T) \not= 0$, there are
$x_1, \ldots,x_{m_0} \in M$ linearly independent with $b(x_i,x_j) = 2t_{ij}$ ($1 \leq i,j
\leq m_0$). The set of such $m_0$-tuples $x_1,\ldots,x_{m_0}$ can be split up into 
$\Pi$-orbits, where each non-trivial orbit has length divisible by $p$. Orbits with one
element (if they exist) consist of tuples $\x=(x_1,\ldots,x_{m_0}) \in M_0^{m_0}$ with
Gram matrix $G_b(\x)=2T$, and for any tuple $\x \in M_0^{m_0}$ the
determinant of the Gram matrix $G_b(\x)$ is divisible by the determinant of $M_0$. If the exceptional 
conditions given in the theorem are not satisfied, $\det M_0$ is divisible by $p$, so the
assertion follows.
 \end{proof}

 \begin{examples}
\begin{enumerate}
\item
 The Leech lattice has automorphisms of orders $2$, $3$, $4$, $5$,
 $7$, $11$, $13$, $23$. Of these, 
the automorphisms
of order $11$ and $23$ do not act fixed point free. Since the Leech lattice is indecomposable, its
theta series $F$ of degree $m_0$ satisfies $\Theta F \equiv 0\bmod p$ for $p = 11$ and for $p = 23$. For
$p = 23$ we have $m_0 = 2$, and it is easily seen that the theta series of degree $2$ itself has
non-degenerate Fourier coefficients which are not divisible by $23$ (the automorphism group of the
binary fixed lattice $M_0$ has order not divisible by $23$). In fact,
from \cite{quebbemann} one sees that $M_0$ has Gram matrix 
\begin{equation*}
  \begin{pmatrix}
    4&1\\
1&6
  \end{pmatrix}.
\end{equation*}

For $p =
11$, inspection of the character table of the group $Co_1$ in
\cite{atlas} shows that
$m_0=4$  holds in this case. Since the order of the automorphism group
of a lattice of rank $4$ can not be divisible by $11$, it is again
clear that the theta 
series of degree $4$ has non degenerate Fourier coefficients not
divisible by $11$ and that $4$ is the largest degree in which the
theta series has this property (whereas the theta series of degree $5$
or higher 
is singular mod $p$ in the terminology of \cite{boe-kik}). 

Ozeki
\cite{ozeki} has recently computed part of the degree $4$ theta series
of the Leech lattice, he found that the Gram matrix 
\begin{equation*}
  \begin{pmatrix}
    4&2&1&0\\
2&4&1&1\\
1&1&4&2\\
0&1&2&4
  \end{pmatrix}
\end{equation*}
 of determinant
$121$ (belonging to the unique non-principal ideal for a maximal order
in the quaternion algebra ramified at $\infty$ and $11$) is represented $12599323656192000$ times by the Leech
lattice (this number is equal to $1/660$ times the order of the
automorphism group $Co_0$ of the Leech lattice). 
Since this representation number is not divisible by $11$ the
lattice generated by a set of representing vectors must be contained
in the fixed lattice $M_0$ of an automorphism of order $11$, and
since this lattice is maximal, it must be equal to $M_0$. We see that
the Gram matrix above is indeed
associated to the fixed lattice of such an automorphism, and the degree
$4$ theta series of the Leech lattice is congruent modulo $11$ to the
theta series of degree $4$ of this quaternary lattice.

\item The automorphisms of order $13$ of the Leech lattice act (see again
the character table in the atlas) fixed point free. If we put 
$M = M_0 \perp M_1$, where $M_0$ is the $E_8$-lattice and $M_1$ is the Leech lattice, we have
$m_0=8$, and the degree $8$ theta series $F$ of $M$ has Fourier coefficient $|O(E_8)|$ at the
Gram matrix $T$ of the $E_8$-lattice, so this coefficient, which is not divisible by $13$, also
appears in $\Theta F$. 

\item Let $M_1$ be any even unimodular positive definite lattice having an automorphism of order
  $p$ (where $p\ne 2$ is prime) which acts fixed point free and let
  $M_0$ be any positive definite even unimodular lattice of rank $m_0$
  whose
  automorphism group is trivial (such lattices are known to exist if
  the rank $m_0$  is at least $144$ (\cite{bannai}).
Then $m_0$ is the largest degree in which the theta series of $M=M_0\perp M_1$ is
not singular modulo $p$ and the degree $m_0$ theta series of
$M$ is not annihilated by the theta operator.
\end{enumerate}
 \end{examples}
 
 \begin{remark}
The degree $m_0$ of the theta series considered in Theorem
\ref{theta_theorem}  is the maximal degree 
for which the non-degenerate Fourier coefficients are not forced to be divisible by the prime 
$p$ (so that the theta series of degree $m_0+1$ and higher are
singular mod $p$).
 \end{remark}
 \begin{remark}
Theorem \ref{lattice_theorem} allows to give a modified version of the
above theorem also for the action of the theta operator on theta 
series of non-unimodular lattices. One has then:

\begin{quote}
$\Theta F_M \equiv 0 \bmod p$ unless $M$ is decomposable (into an
orthogonal sum of sublattices) and has a proper orthogonal
summand of rank $m_0$ and determinant prime to $p$. 
\end{quote}

Since this
formulation loses the simple criterion on the dimension $m_0$ of the
fixed space and since most of the other work on this type of problem
has been concerned with Siegel modular forms for the full modular
group and hence with even unimodular lattices, we chose to concentrate
on this case in the Theorem above.
\end{remark}

\medskip
Rainer Schulze-Pillot\\
Fachrichtung 6.1 Mathematik,
Universit\"at des Saarlandes (Geb. E2.4)\\
Postfach 151150, 66041 Saarbr\"ucken, Germany\\
email: schulzep@math.uni-sb.de
\end{document}